\documentclass[a4paper,12pt,draft]{amsart}
\usepackage[margin=30mm]{geometry}
\usepackage{amssymb, amsmath, amsthm, amscd, verbatim}
\usepackage{enumerate}
\usepackage[T1]{fontenc}
\usepackage{eucal,mathrsfs,dsfont}
\usepackage{color}
\usepackage{mathtools}
\usepackage{marginnote}

\definecolor{mno}{rgb}{0.5,0.1,0.5}

\newcommand{\R}{\mathds R}

\newcommand{\rd}{{\mathds R^d}}
\newcommand{\Pp}{\mathds P}
\newcommand{\Ee}{\mathds E}

\newcommand{\C}{\mathds C}
\newcommand{\I}{\mathds 1}

\newcommand{\Dh}{{\mathrm{dim}}_{\mathrm H}}
\newcommand{\capp}{\mathrm{Cap}}

\newcommand{\diam}{\operatorname{diam}}
\newcommand{\hausd}{\mathcal{H}}
\newcommand{\betaupperstar}{\beta_\infty}
\newcommand{\betalower}{\delta_\infty}

\newtheorem{theorem}{Theorem}[section]
\newtheorem{lemma}[theorem]{Lemma}

\theoremstyle{definition}

\newtheorem{example}[theorem]{Example}
\newtheorem{remark}[theorem]{Remark}

\renewcommand{\le}{\leqslant}
\renewcommand{\ge}{\geqslant}
\renewcommand{\leq}{\leqslant}
\renewcommand{\geq}{\geqslant}
\renewcommand{\Re}{\ensuremath{\operatorname{Re}}}
\renewcommand{\Im}{\ensuremath{\operatorname{Im}}}
\newtheorem*{ack}{Acknowledgement}

\begin{document}
\allowdisplaybreaks

\title[Lower Bounds of the Hausdorff dimension for Feller processes]{\bfseries Lower Bounds of the Hausdorff dimension for the images of Feller processes}

\author{V.\ Knopova \and R.L.\ Schilling \and J.\ Wang}
\thanks{\emph{V.\ Knopova:}
V.\ M.\ Glushkov Institute of Cybernetics NAS of Ukraine, 30187, Kiev, Ukraine.
 \texttt{vic\underline{ }knopova@gmx.de}}
\thanks{\emph{R.L.\ Schilling:} TU Dresden, Institut f\"{u}r Mathematische Stochastik, 01062 Dresden, Germany. \texttt{rene.schilling@tu-dresden.de}}
\thanks{\emph{J.\ Wang:}
School of Mathematics and Computer Science, Fujian Normal
University, 350007, Fuzhou, P.R. China.
\texttt{jianwang@fjnu.edu.cn}}
\date{}

\begin{abstract} Let $(X_t)_{t\ge0}$ be a Feller process generated by a pseudo-differential operator
whose symbol satisfies $\|p(\cdot,\xi)\|_\infty\le c(1+|\xi|^2)$ and $p(\cdot,0)\equiv0.$ We prove that, for a large class of examples, the Hausdorff dimension of the set $\{X_t: t\in E\}$ for any analytic set $E\subset [0,\infty)$ is almost surely bounded below by  $\betalower \Dh E$,
where
\begin{align*}
       \betalower&:=\sup\left\{\delta>0: \lim_{|\xi|\to \infty} \frac{\inf_{z\in\R^d} \Re p(z,\xi)}{|\xi|^\delta}=\infty\right\}.
\end{align*}This, along with the upper bound $ \betaupperstar \Dh E$ with \begin{align*}
   \betaupperstar
    &:=\inf\left\{\delta>0: \lim_{|\xi|\to \infty}\frac{\sup_{|\eta|\le {|\xi|}}\sup_{z\in\R^d} |p(z,\eta)|}{|\xi|^\delta}=0\right\}
    \end{align*} established in B\"{o}ttcher, Schilling and Wang (2014), extends the dimension estimates for L\'{e}vy processes of Blumenthal and Getoor (1961) and Millar (1971) to Feller processes.


\end{abstract}

\maketitle

\section{Background and Main Result}\label{section1}
A \emph{Feller process} $(X_t)_{t\ge0}$ with state space $\R^d$ is a strong
Markov process such that the associated operator semigroup $(T_t)_{t\ge0}$,
$$
    T_tu(x)=\Ee^x(u(X_t)),\qquad u\in C_\infty(\R^d),\; t\ge0,\; x\in\R^d,
$$
($C_\infty(\R^d)$ is the space of continuous functions
vanishing at infinity) enjoys the Feller property, i.e.\ it maps
$C_\infty(\R^d)$ into itself. A semigroup is said to be a
\emph{Feller semigroup}, if $(T_t)_{t\ge0}$ is a
one-parameter semigroup of linear contraction operators
$T_t:C_\infty(\R^d)\rightarrow C_\infty(\R^d)$ which is strongly
continuous: $\lim_{t\to0}\|T_tu-u\|_\infty=0$ for any $u\in C_\infty(\R^d)$, and has the sub-Markov
property: $0\le T_tu\le 1$ whenever $0\le u\le 1$. The infinitesimal
generator $(A,{D}(A))$ of the semigroup $(T_t)_{t\ge0}$ (or of the
process $(X_t)_{t\ge0}$) is given by the strong limit
$$
    Au:=\lim_{t\to0}\frac{T_tu-u}{t}
$$
on the set ${D}(A)\subset C_\infty(\R^d)$ of all $u\in C_\infty (\R^d)$ for
which the above limit exists with respect to the uniform norm.

Let $C_c^\infty(\R^d)$ be the space of smooth functions with compact
support. Under the assumption that the test functions
$C_c^\infty(\R^d)$ are contained in $D(A)$, Ph.\ Courr\`{e}ge
\cite[Theorem 3.4]{Cou} proved that the generator $A$ restricted to
$C_c^\infty(\R^d)$ is a pseudo-differential operator,
\begin{equation*}\label{symbol}
    Au(x)
    =-p(x,D)u(x)
    :=-\int e^{i \langle x,\xi\rangle}\,p(x,\xi)\,\hat{u}(\xi)\,d\xi,\quad u\in C_c^\infty(\R^d),
\end{equation*}
with \emph{symbol} $p:\R^d \times \R^d\to \C$, where $\hat{u}$ is
the Fourier transform of $u$, i.e.\ $\hat{u}(x)=(2\pi)^{{-d}}\int
e^{-i\langle x,\xi\rangle}u(\xi)\,d\xi$. The symbol $p(x,\xi)$ is locally
bounded in $(x,\xi)$, measurable as a function of $x$, and for every
fixed $x\in\R^d$ it is a continuous negative definite function in
the co-variable. This is to say that it enjoys the following
L\'{e}vy-Khintchine representation,
\begin{equation}\label{sy}
    p(x,\xi)
    =c(x)-i\langle b(x),\xi\rangle+\frac{1}{2}\langle\xi,a(x)\xi\rangle
    +\int\limits_{\mathclap{z\neq 0}}\!\!\big(1-e^{i\langle z,\xi\rangle}+i\langle z,\xi\rangle\I_{\{|z|\le1\}}\big)\,\nu(x,dz),
\end{equation}
where $(c(x),b(x),a(x),\nu(x,dz))_{x\in\R^d}$ are the L\'{e}vy
characteristics:  $c(x)$ is a nonnegative measurable function,
$b(x):=(b_j(x))\in\R^d$ is a measurable function, $a(x):=(a_{jk}(x))\in \R^{d\times d}$ is a nonnegative definite
matrix-valued function, and $\nu(x,dz)$ is a nonnegative, $\sigma$-finite kernel on
$\R^d\times\mathscr{B}(\R^d\setminus\{0\})$ such that for every
$x\in\R^d$, $\int_{z\neq 0}(1\wedge |z|^2)\,\nu(x,dz)<+\infty$. For details and a comprehensive bibliography we refer to the monographs \cite{jac-book} by N.\ Jacob and the survey \cite{BSW}. Since we will only consider the case where $c\equiv 0$, we will from now on use the L\'evy triplet $(b(x),a(x),\nu(x,dz))$.

It is instructive to have a brief look at L\'{e}vy processes which are a particular but important subclass of Feller processes. Our standard
reference for L\'{e}vy processes is the monograph by K.\ Sato
\cite{SA}. A \emph{L\'{e}vy process} $(Y_t)_{t\ge0}$ is a stochastically
continuous random process with stationary and independent
increments. The characteristic exponent (or symbol) $\psi:\R^d\to \C$ of a L\'{e}vy process is a continuous negative definite function, i.e.\ it is given by a L\'evy-Khintchine formula of the form \eqref{sy} with characteristics $(b,a,\nu(dz))$ which do not depend on $x$.

The notion of Hausdorff dimension is very useful in order to characterize the irregularity of stochastic processes.
The Hausdorff dimension of the image sets of a L\'{e}vy process has been extensively studied, see \cite{BG, P, Mu} and also the survey papers \cite{Ta, Xiao} for details. Recall that the \emph{Hausdorff dimension} of a set $A\subset \R^d$ is the unique number $\lambda$, where the $\lambda$-dimensional \emph{Hausdorff measure} $\hausd^\lambda(A)$, defined by
\begin{gather*}
   \hausd^\lambda(A)
    =\sup_{\varepsilon>0}\inf\left\{\sum_{n=1}^\infty\left(\diam A_n\right)^\lambda: A_n \text{\ Borel, \ } \bigcup_{n=1}^\infty A_n \supset A\text{\ and \ } \diam A_n\le \varepsilon\right\},
\end{gather*}
changes from $+\infty$ to a finite value.

For the study of the Hausdorff dimension for the sample paths of L\'{e}vy processes, various indices were introduced in
\cite[Sections 2, 3 and 5]{BG}:
\begin{align*}
          \beta'' &=\sup\bigg\{\delta>0: \lim_{|\xi|\to \infty} \frac{\Re \psi(\xi)}{|\xi|^\delta}=\infty\bigg\},\\
          \beta  &=\inf\bigg\{\delta>0: \lim_{|\xi|\to \infty}\frac{|\psi(\xi)|}{|\xi|^\delta}=0\bigg\}.
\end{align*}
The results for the Hausdorff dimension of the image sets of L\'{e}vy processes can be summarized as follows, see \cite{BG, P, Mu}:

\begin{remark}\label{theo1} Let $(Y_t)_{t\ge0}$ be
a $d$-dimensional L\'{e}vy process with indices $\beta''$ and $\beta$ given above. For every analytic set $E\subset [0, 1]$ we have, almost surely
$$
    \min\{d,\beta''\Dh E\}
    \le \Dh  Y(E)
    \le \min\{d, \beta \Dh E\}.
$$
\end{remark}

Now, we turn to the Hausdorff dimension for the image of a Feller processes. Throughout we will make the following assumptions on the symbol $p(x,\xi)$:
\begin{equation}\label{assumption}
    \|p(\cdot,\xi)\|_\infty\le c(1+|\xi|^2)
    \qquad\text{and}\qquad
    p(\cdot,0)\equiv0.
\end{equation}
The first condition means that the generator has only bounded `coefficients', see, e.g.\ \cite[Lemma 2.1]{S3} or \cite[Lemma 6.2]{Sch}; the second condition implies then (if the first condition is satisfied) the Feller process is conservative in the sense that the life time of the process is almost surely infinite, see \cite[Theorem
5.2]{S1}.

\medskip

We first recall the following upper bound for the Hausdorff dimension for the image of a Feller processes, which partly extends Remark \ref{theo1}.

\begin{theorem}\label{pro1}{\bf (\cite[Theorem 5.15]{BSW})}
    Let $(X_t)_{t\ge0}$ be a Feller process with the generator $(A,{D}(A))$ such that $C_c^\infty(\R^d)\subset D(A)$, i.e.\ $A|_{C_c^\infty(\R^d)}=-p(\cdot,D)$ is a pseudo-differential operator with symbol $p(x,\xi)$. Assume that the symbol satisfies \eqref{assumption}. Then, for every bounded analytic set $E\subset[0,\infty)$,
\begin{equation}\label{pro11}
     \Dh  X(E)\le \min\{d, \betaupperstar \Dh E\}
\end{equation}
    holds almost surely, where the generalized Blumenthal--Getoor indices (at infinity) are given by
\begin{align*}
    \betaupperstar
    &:=\inf\bigg\{\delta>0: \lim_{|\xi|\to \infty}\frac{\sup_{|\eta|\le {|\xi|}}\sup_{z\in\R^d} |p(z,\eta)|}{|\xi|^\delta}=0\bigg\}.
    \end{align*}
\end{theorem}
\begin{remark}\label{remark2} (1) The index $\betaupperstar$ coincides with the index $\beta_\infty^{\R^d}$ from \cite[Remark\,5.14b]{BSW}.
(2) The  inequality  \eqref{pro11}  with $E= [0,1]$ can  also be  deduced from the variation of sample functions for Feller processes. Recall that, for a $p\in(0,\infty)$ and a function $f$ defined on the interval $[0,T]$ and taking values in $\R^d$, its $p$-variation is given by
$$
    V_p(f, [0,T]) = \sup\bigg\{\sum_{j=0}^{m-1}|f(t_{j+1})-f(t_j)|^p: 0=t_0< t_1<\cdots<t_m=T,\,\,m\ge1\bigg\},
$$where the supremum is taken over all finite subdivisions $0=t_0< t_1<\cdots<t_m=T.$
   Let $(X_t)_{t\ge0}$ be a Feller process given in Theorem \ref{pro1}. Then, for any $p>\betaupperstar$,
    $$
        \hausd^p X([0,1])\le 2^{p} V_p(X, [0,1])<\infty
    $$
    holds almost surely, which implies that $\Dh X([0,1])\le \betaupperstar.$
\end{remark}

\medskip

The purpose of this paper is to establish the lower bound estimate for the Hausdorff dimension for the image of a Feller process. For this we need more assumptions on the Feller process. The main statement is as follows.
\begin{theorem}\label{pro2}
Let $(X_t)_{t\ge0}$ be a Feller process in $\rd$ with the transition probability density $p(t,x,y)$, which satisfies
\begin{equation}\label{diag1}
    p(t,x,y)\leq c t^{-d/\alpha},  \quad t\in (0,1], \quad x,y\in \rd,
\end{equation}
for some $\alpha\in (0,2)$.
Then,  for any analytic set $E\subset [0,1]$ we have
\begin{equation}\label{HD1}
    \Dh X(E) \geq \big( \alpha \wedge d  \big)\Dh E.
\end{equation}
\end{theorem}

Let us give a few examples where the conditions of Theorem \ref{pro2} are satisfied.
\begin{example}\label{exa}
\begin{enumerate}[a)]
\item
For a L\'evy process condition \eqref{diag1} is easy to check; for example, it holds true that the characteristic exponent $\psi$ satisfies
\begin{equation}\label{re-p0}
    \Re \psi(\xi)\geq c |\xi|^\alpha,\quad |\xi| > 1.
\end{equation}

\item
For a symmetric Markov process condition \eqref{diag1}
is equivalent to the following Nash type inequality
$$\|f\|_{L^2(\R^d;dx)}^{2+2\alpha/d}\le C\Big[ D(f,f)+\delta \|f\|_{L^2(\R^d;dx)}^2\Big]\|f\|_{L^1(\R^d;dx)}^{2\alpha/d},\quad f\in C_c^\infty(\R^d),$$ for some positive constants $C$ and $\delta$,  where $D(f,f)=-\langle f,Af\rangle_{L^2(\R^d;dx)}$. See \cite[Vol.\ II, Section~3.6]{jac-book}, also the original paper \cite{CKS87} and  \cite{BM07,SW12} for more recent developments. We refer to \cite[Proposition II.1]{TC} for a proof (relying on functional inequalities) that general non-symmetric semigroups satisfy \eqref{diag1}.

\item
Sufficient conditions when a L\'{e}vy type  process satisfies \eqref{diag1} are given in \cite{KK13a, KK13b}, where the  approach relies on the parametrix construction of a Markov process.

Consider the triplet
\begin{equation}\label{lh}
    (b(x), 0, m(x,z)\,\mu(dz))
\end{equation}
where the functions $b(\cdot)$ and $m(\cdot,z)$ are bounded and H\"older continuous with $m(x,z)\geq c>0$ for all $x$, $z\in\R^d$; $\mu$ is a  L\'evy measure: $\int_\rd (1\wedge|z|^2)\, \mu(dz)<\infty$, moreover  it satisfies the following condition: There exists $\beta>1$ such that
\begin{equation}\label{ul}
\sup_{\ell \in \mathds{S}^d} q^U(r \ell) \leq \beta \inf_{\ell \in \mathds{S}^d} q^L(r  \ell), \quad r\geq 1,
\end{equation}
where $\mathds{S}^d$ is the unit sphere in $\rd$, and
  $$
  q^U(\xi):= \int_\rd (1\wedge\langle\xi, z \rangle^2 ) \mu(dz), \quad q^L(\xi) := \int_{|\langle\xi, z \rangle|\leq 1} \langle\xi, z \rangle^2 \mu(dz).
  $$
    Note that the  L\'evy--Khintchine exponent
$$
    q(\xi)= \int_{\R^d} (1-\cos \langle\xi, z \rangle) \mu(dz),
$$
always satisfies the inequalities $ (1-\cos 1) q^L(\xi)\leq q(\xi) \leq 2 q^U(\xi)$; moreover, we have for large $|\xi|$ the relations
$$
    q(\xi) \asymp q^L(\xi)\asymp q^U(\xi),\quad |\xi|>1,
$$
i.e.\ the exponent does not oscillate ``too much''. In particular, condition \eqref{ul} on the  L\'evy measure $\mu(dz)$ together with the boundedness of $m(x,z)$ implies that
\begin{equation}\label{re-p}
\inf_{x\in \rd} \Re p(x,\xi)\geq c |\xi|^\alpha, \quad |\xi| > 1,
\end{equation}
holds for  $\alpha= 2/\beta \in (0,2)$ and $c>0$, see \cite{K13, KK13b}.     It was proved in \cite{KK13b} that starting with the L\'evy  characteristics as in \eqref{lh}   there exists a Feller process  $(X_t)_{t\ge0}$  associated with symbol $p(x,\xi)$ given by \eqref{sy} and, moreover, this process $(X_t)_{t\ge0}$ possesses a transition probability density $p(t,x,y)$. 
Similar conditions (for a slightly different L\'evy triplet) are given in \cite{KK13a}.

\item
Condition \eqref{diag1} holds true for the transition probability density of the stable-like process  associated  with the L\'{e}vy triplet $(0,0, |z|^{-1-\alpha(x)} \,d|z|\,\tilde{\mu}(x,d\ell))$, where $\ell:= z/|z|$ for $z\in \rd$; here, the index function $\alpha(x)$  and the kernel $\tilde{\mu}(x,d\ell)$  are bounded and continuous such that
$$
C_1\leq \int_{\mathds{S}^d} |(v,\ell)|^{\alpha(x)} \tilde{\mu}(x,d\ell) \leq C_2,\quad v\in \mathds{S}^d, x\in \R^d,
$$
holds with some positive constants $C_1$ and $C_2$, see  \cite[Theorem~5.1]{Ko00}.  In this case, $\alpha$ in \eqref{diag1} is equal to  $\min_{x\in\R^d} \alpha(x)$.
\end{enumerate}
\end{example}

\begin{remark}
Note that in the cases a), c) and d)  (cf. Example~\ref{exa})  the characteristic exponent and the symbol satisfy, respectively, \eqref{re-p0} and \eqref{re-p} where, for d),  $\alpha= \min_{x\in\R^d} \alpha(x)$. This allows to state the lower bound in terms of the Blumenthal--Getoor index:
$$
\Dh X(E)\geq \big(\betalower \wedge d\big) \Dh E,
$$
 where
$$
 \betalower
   :=\sup\bigg\{\delta>0: \lim_{|\xi|\to \infty} \frac{\inf_{z\in\R^d} \Re p(z,\xi)}{|\xi|^\delta}=\infty\bigg\}
$$
is the generalized Blumenthal--Getoor index at infinity.   This index coincides with the index $\delta_\infty^{\R^d}$ from \cite[Remark\, 5.14b]{BSW}.
\end{remark}

\section{Proofs}
 We begin with the
\begin{proof}[Proof of Remark~$\ref{remark2}$ $(2)$]
According to \cite[Theorem 1.3]{M1} and \cite[Theorem 3]{M2}, we know that \emph{if there exist two positive constants $r_0$ and $C$ such that for all $t>0$ and $0<r<r_0$,
$$
    a(t,r):\,=\sup_{0<s\le t,\,x\in\R^d}\Pp^x\big(|X_s-x|\ge r\big)\le Ct^\beta r^{-\alpha}
$$
with two constants $\alpha>0$ and $\beta>(3-e)/(e-1)$, then for any $p>\alpha/\beta$, the $p$-variation of the sample function $(X_t)_{t\ge0}$ is finite almost surely.}
It is clear that for any $t$, $r>0$,
$$
    a(t,r)\le \sup_{x\in \R^d} \Pp^x\big(\sup_{s\le t}|X_s-x|\ge r\big).
$$
Applying \cite[Corollary\,5.2]{BSW} yields that
$$
    a(t,r)
    \le c\,t\, \sup_{x\in\R^d} \sup_{|\xi|\le 1/r}|p(x,\xi)|
$$
holds for some constant $c>0$. By the very definition of $\betaupperstar$, for any $p>\betaupperstar$, there exists an $r_{0}$ small enough such that for any
$0<r\le r_{0}$,
$$
    \sup_{x\in\R^d} \sup_{|\xi|\le 1/r}|p(x,\xi)|\le r^{-p},
$$
which proves the finiteness of the $p$-variation because of Manstavi\v{c}ius' results mentioned earlier.
The remaining part of the proof immediately follows from \cite[Theorem 8.4]{BG}.
\end{proof}
The proof of  Theorem~\ref{pro2} relies on several results,  which for the reader's convenience  we quote below. First, define the $\lambda$-capacity of a Borel set $B\subset \rd$ as follows:
\begin{equation}
    \capp_\lambda (B):= \left(\inf\Big\{ \int_B\int_B |x-y|^{-\lambda}\varpi(dy)\varpi(dx)\::\: \varpi\in\mathcal M_1^+(B)\Big\}\right)^{-1} \label{cap}
\end{equation}
where $\mathcal M_1^+(B)$ denotes the probability measures with support in $B$.
We need the following version of Frostman's lemma, see \cite[Theorem 4.13]{Fa03} or the original paper   \cite{Fr35}.
\begin{lemma}\label{Fr} If $F\subset \rd$ is a closed set with strictly positive Hausdorff measure $\mathcal{H}^\lambda(F)>0$ for some $\lambda>0$, then $\capp_{\lambda'}(F)>0$ for all $\lambda'<\lambda$.
\end{lemma}

The lemma below is taken from \cite[Lemma~2.2]{BG}, see also \cite{Mc}.
\begin{lemma} \label{rene-lem2}
Let $f: \mathcal{X}\to \rd$ be a measurable function on a metric space $(\mathcal{X},d(\cdot,\cdot))$, and $E\subset \mathcal{X}$ be a Borel set.  If there exists a probability measure $\varpi\in \mathcal M_1^+(E)$ such that
\begin{equation}
\int_E\int_E |f(x)-f(y)|^{-\lambda}\varpi(dx)\varpi(dy)<\infty\label{mu1}
\end{equation}
for some $\lambda>0$, then  $\mathcal{H}^\lambda(f(E))>0$.
\end{lemma}
\medskip

Let $(Y_t)_{t\ge0}$ be a Markov process in $\rd$, and
\begin{equation}
\beta'(Y,x):= \sup\left\{\lambda\geq 0:\,\, \Ee^x\big( |Y_t-Y_s|^{-\lambda}\big) =O (|t-s|^{-1}) \quad \text{as $t-s\to 0$}\right\}, \label{bet}
\end{equation}
This index was introduced in \cite{Sch94}.
\begin{lemma}\label{rene-lem1}
Let $(Y_t)_{t\ge0}$ be a Markov process with values in $\rd$, and $E\subset [0,1]$ be an  analytic set with Hausdorff dimension $\Dh E$. Then,
\begin{equation}
\Dh Y(E) \geq \beta'(Y,x)\Dh  E \quad \Pp^x\text{-a.s.}\label{lem-dim}
\end{equation}
\end{lemma}
\begin{proof}
Let  $0<\lambda<\beta'(Y,x)$ and $0<\alpha<\alpha'<\Dh E$.  We find, using Jensen's inequality for concave functions, that there exists a constant $C>0$ such that
\begin{equation}\label{rene1}
\Ee^x |Y_t-Y_s|^{-\lambda\alpha} \leq  \left(\Ee^x |Y_t-Y_s|^{-\lambda}\right)^\alpha \leq C |t-s|^{-\alpha}, \quad |t-s|\leq 1.
\end{equation}
Since $\alpha'<\Dh E$, we have $\mathcal{H}^{\alpha'}(E)=\infty$, where $\mathcal{H}^{\alpha'}(E)$ is the Hausdorff measure  of $E$ with dimension $\alpha'$. We will use the following result, proved in \cite[p.~489, Corollary]{Da}: \emph{Let  $A\subset \rd$ be an analytic set with $\mathcal{H}^\lambda(A)=\infty$ for some $\lambda>0$. Then for any $r>0$  there exists a closed subset $F_r\subset A$ such that $\mathcal{H}^\lambda (F_r)\geq r$.} By this, there exists a closed set $F\subset E$ such that $\mathcal{H}^{\alpha'}(F)>0$, which implies by Lemma~\ref{Fr} that $\capp_\alpha(F)>0$  for all $\alpha<\alpha'$.
By the definition of the capacity, there exists a probability measure with support on $F$ such that
$$
\int_F \int_F |t-s|^{-\alpha} \mu(dt)\mu(ds)<\infty.
$$
Thus, the above inequality, together with \eqref{rene1} and the Fubini theorem, gives us
$$
\Ee^x \left( \int_F \int_F |Y_t-Y_s|^{-\alpha \lambda} \mu(dt)\mu(ds)\right) <\infty,
$$
which in turn yields
$$
\int_F \int_F |Y_t-Y_s|^{-\alpha\lambda} \mu(dt)\mu(ds)<\infty \quad \Pp^x\text{-a.s.}
$$
According to Lemma~\ref{rene-lem2},  we get
$$
\Lambda^{\alpha\lambda} (Y(E,\omega))\geq \Lambda^{\alpha \lambda} (Y(F,\omega))>0 \quad \Pp^x\text{-a.s.}
$$
Therefore, we derive the statement of the lemma by passing to the limit as $\alpha\uparrow \Dh E$ and $\lambda \uparrow \beta'(Y,x)$.
\end{proof}

Now, we are in a position to present the proof of Theorem \ref{pro2}.

\begin{proof}[Proof of Theorem $\ref{pro2}$]
By the strong Markov property, we have for any $\lambda<d\wedge\alpha$, $x\in \rd$ and $0<s<t$ with $t-s\leq 1$
\begin{align*}
    \Ee^x |X_t-X_s|^{-\lambda}
    &=\Ee^x \left[ \Ee^{X_s} |X_{t-s}-X_0|^{-\lambda}\right]
    \le \sup_{z\in\R^d}\Ee^{z} |X_{t-s}-z|^{-\lambda}\\
    &=\sup_{z\in\R^d} \Ee^{z} \left[\int_0^\infty\I_{\{|X_{t-s}-z|^{-\lambda}>u\}}\,du\right]\\
    &\le \sup_{z\in\R^d} \Ee^{z} \left[\int_{(t-s)^{-\lambda/\alpha}}^\infty\I_{\{|X_{t-s}-z|^{-\lambda}>u\}}\,du\right]+ (t-s)^{-\lambda/\alpha}\\
    &= \lambda \sup_{z\in\R^d} \Ee^{z} \left[\int_0^{(t-s)^{1/\alpha}} u^{-\lambda-1}\I_{\{|X_{t-s}-z|<u\}}\,du\right] + (t-s)^{-\lambda/\alpha}\\
    &\le \lambda \int_0^{(t-s)^{1/\alpha}} u^{-\lambda-1} \sup_{z\in\R^d} \Pp^z (|X_{t-s}-z|\le u) \,du + (t-s)^{-\lambda/\alpha}.
\end{align*}
According to \eqref{diag1}, there is a constant $C>0$ such that for all $0<u\leq (t-s)^{1/\alpha}$,
$$
    \sup_{z\in\R^d} \Pp^z (|X_{t-s}-z|\le u)\le C (t-s)^{-d/\alpha}u^d.
$$
This, along with the inequality above and the fact that $\lambda<d\wedge\alpha$, yields that
$$
    \Ee^x |X_t-X_s|^{-\lambda}
    \le \left(\frac{C\lambda}{d-\lambda}+1\right)(t-s)^{-\lambda/\alpha}
    \le \left(\frac{C\lambda}{d-\lambda}+1\right)(t-s)^{-1} .
$$
Thus, we have the lower estimate for the index $\beta'(X,x)$ defined in \eqref{bet}:
$$
    \beta'(X,x)\geq  \lambda \quad \Pp^x\text{-a.s.}
$$
Letting $\lambda\to d\wedge\alpha$ in the inequality above, we prove \eqref{HD1} by Lemma \ref{rene-lem1}.
\end{proof}

\begin{remark}
    The proof of Theorem~\ref{pro2} shows that we can replace the condition \eqref{diag1} by the following assumptions
    \begin{gather}\label{diag2}
        \exists c\ge 1,\;
        \forall z\in\rd,\; 0<u,t\le 1\::\:
        \Pp^z\left(|X_{t}-z|\leq u\right)
        \leq c\Pp^z\left(\sup_{0<r\leq t}|X_r-z|\leq cu\right),\\
        \label{diag3}\exists\kappa>0,\;\forall x,\xi\in\rd\::\: |\Im p(x,\xi)| \leq \kappa \Re p(x,\xi).\quad\text{(sector condition)}
    \end{gather}
    If \eqref{diag2} and \eqref{diag3} are satisfied, we can replace $\alpha$ in the proof by the index $\delta^*_\infty$ defined by 
    \begin{align*}
      \delta^*_\infty&:=\sup\left\{\delta>0: \lim_{|\xi|\to \infty} \frac{\sup_{|\eta|\le |\xi|}\inf_{z\in\R^d} \Re p(z,\eta)}{|\xi|^\delta}=\infty\right\},
\end{align*}
  and we find using the second maximal estimate for Feller processes, cf.\ \cite[Corollary 5.6, p.\ 116]{BSW}, for $\lambda < \lambda' < d\wedge\delta^*_\infty$ and $0<s<t$ with $t-s<1$
    \begin{align*}
        \Pp^z\left(\sup_{0<r\leq t-s}|X_r-z|\leq cu\right)
        \leq c'\left((t-s) \sup_{|\xi|\leq 1/(2\kappa cu)}\inf_{y\in\rd} \Re p(y,\xi)\right)^{-1}
        \leq \frac{c'' u^{\lambda'}}{(t-s)}.
    \end{align*}
    Thus we get that for all $x\in\R^d$ and $0<s<t$ with $t-s<1$
    $$
        \Ee^x |X_t-X_s|^{-\lambda} \leq \left(\frac{cc''\lambda}{\lambda'-\lambda}+1\right) (t-s)^{-1},
    $$
    and the rest of the argument is as before.
\end{remark}
\begin{ack}We thank the referee for
careful reading of this paper.
Financial support through the Scholarship  for Young Scientists
2012-2014, Ukraine (for Victoria Knopova), DFG (grant Schi 419/8-1)
(for Ren\'{e} L.\ Schilling), and the National Natural Science
Foundation of China (No.\ 11201073) and the
Program for Nonlinear Analysis and Its Applications (No.\ IRTL1206)
(for Jian Wang) are gratefully acknowledged.
\end{ack}

\end{document}